\documentclass[12pt]{article}
\usepackage{amssymb, amsmath,epsfig,amsthm}
\begin{document}
 \righthyphenmin=2

\def \<{{\langle}}
\def\>{{\rangle}}
\def\Z{{\Bbb Z}}
\def\N{{\mathbb N}}
\def\NN{{\mu_1}}
\def\T{{\cal T}}
\def \D{{\cal D}}
\def \C{{\cal C}}
\def \B{{\cal B}}
\def \Di{{\cal D}}
\def \di{{\Bbb D}}
\def \og{{\Bbb O}}
\def \H{{\cal H}}
\def \I{{\cal I}}
\def \Tht{{\Theta}}
\def \F{{\cal F}}
\def \V{{\cal V}}
\def \a{{\alpha}}
\def \l{{\lambda}}
\def \t{{\beta}}
\def \bt{{\beta}}
\def \dt{{\delta}}
\def \De{{\Delta}}
\def \ga{{\gamma}}
\def \kp{{\varkappa}}
\def \sgn{{\operatorname{sgn}}}
\def \R{{\mathbb R}}
\def \CR{{\cal R}}
\def \A {{\cal A}}
\def \a{{\alpha}}
\def \m{{\mu}}
\def \M{{\cal M}}
\def \g{{\gamma}}
\def \bt{{\beta}}
\def \la{{\lambda}}
\def \La{{\Lambda}}
\def \fee{{\varphi}}
\def \e{{\epsilon}}
\def \s{{\sigma}}
\def \Sm{{\Sigma}}
\def \sm{{\sigma_m}}
\def \PLb{{P^\bot_L}}
\def \td{\tilde}
\newcommand{\wh}{\widehat}
\newcommand{\wt}{\widetilde}
\newcommand{\ol}{\overline}
\def \supp{{\operatorname{supp}}}
\def \sign{{\operatorname{sign}}}
\def \sp{{\operatorname{span}}}
\def \csp{{\overline{\operatorname{span}}}}
\def \Prj{{\operatorname{Proj}}}
\def \Dff{{\operatorname{Diff}}}
\def \Dst{{\operatorname{Dist}}}
\def \argmax{{\operatorname{argmax}}}
\def \argmin{{\operatorname{argmin}}}

\def\charf\#1{\chi_{\lower3pt\hbox{$\scriptstyle \#1$}}} 
\catcode`\@=11
\def\NoLogo{\let\logo@\empty}
\catcode`\@=\active \NoLogo

\newtheorem{Theor}{Theorem}
\newtheorem{Lemm}{Lemma}
\newtheorem{Corollar}{Collary}
\theoremstyle{remark}
\newtheorem{Remar}{Remark}

\begin{center}
{\Large On efficiency of Orthogonal Matching
Pursuit.}\footnote{This research is partially supported by Russian
Foundation for Basic Research
project 08-01-00799 and 09-01-12173} \\
 Eugene Livshitz\footnote{E-mail: evgliv at gmail.com}
\end{center}
\begin{abstract}
We show that if a matrix $\Phi$ satisfies the RIP of order
$[CK^{1.2}]$ with isometry constant $\dt = c K^{-0.2}$ and has
coherence less than $1/(20 K^{0.8})$, then Orthogonal Matching
Pursuit (OMP) will recover $K$-sparse signal $x$ from $y=\Phi x$
in at most $[CK^{1.2}]$ iterations. This result implies that
$K$-sparse signal can be recovered via OMP by $M=O(K^{1.6}\log N)$
measurements.
\end{abstract}

\section{Introduction.}

The emerging theory of Compressed Sensing (CS) has provided a new
framework for signal acquisition \cite{B}, \cite{C}, \cite{D}. Let
us recall some basic concepts of CS. Let $\Phi$ be a $M\times N$
matrix ($M<N$). The basic problem in CS is to construct a stable
and fast algorithm for recovery a signal $x\in\R^d$ that has $K$
non-zero components ($K$-sparse signal) from measurements $y =
\Phi x\in\R^M$ and to determine $(M,N,K)$ for which such
algorithms exist.

E.~Cand\'{e}s and T.~Tao proved that Basic Pursuit (BP)
\begin{equation*}
\wh x (y) = \argmin\{|z|_1 :\ \Phi z = y \}.
\end{equation*}
can provide the exact recovery of arbitrary $K$-sparse $x\in\R^N$
by $M=O(K\log(N/K))$ measurements.

In this article we study signal recovery via Orthogonal Matching
Pursuit(OMP). Although theoretical results for OMP are essentially
worse than for BP, its computational simplicity allows OMP to
achieve very good result in practise~\cite{Tr}.

\vskip 1cm
 \hrule
\smallskip
 {\bf Algorithm:} Orthogonal Matching Pursuit
\smallskip
 \hrule
 \smallskip
 {\bf Input:} $\Phi$, $y$.

{\bf Initialization:} $r^0:=y$, $x^0:=0$, $\La^0=\emptyset$,
$l=0$.

 {\bf Iterations:}
Define
\begin{equation*}
\La^{l+1} := \La^l\cup\argmax_i |\<r^l, \phi_i\>|,
\end{equation*}
\begin{equation*}
x^{l+1} := \argmin_{z:\ \supp(z)\subset \La^{l+1}} \|y -\Phi
z\|\quad r^{l+1} :=y-\Phi x^{l+1}.
\end{equation*}

 If $r_{l+1} = 0$, stop.
Otherwise let $l := l + 1$ and begin a new iteration.

{\bf Output}: If algorithm stops at $l$-th iteration, output is
$\wh x= x_l$.
\smallskip
 \hrule
\vskip 1cm

By $\phi_i$, $1\le i\le N$, we denote the $i$-th column of $\Phi$.
We assume that $\|\phi_i\|=1$, $1\le i\le N$. To formulate results
on recovery via OMP we use two basic properties of matrix $\Phi$.

\begin{itemize}
\item Coherence of $\Phi$
\begin{equation*}
\mu(\Phi):=\sup_{i\ne j}|\<\phi_i,\phi_j\>|.
\end{equation*}
\item Restricted Isometry Property (\cite{CT}). A matrix $\Phi$
satisfies Restricted Isometry Property (RIP) of order $K$ with
isometry constant $\dt\in(0,1)$ if the inequality
\begin{equation*}
(1-\dt)\|x\|^2\le \|\Phi x\|^2\le (1+\dt)\|x\|^2
\end{equation*}
holds for all $K$-sparse $x\in\R^N$.
\end{itemize}

It's well known (see \cite{GMS}, \cite{GN}) that if
\begin{equation*}
\mu(\Phi)<\frac{1}{2K-1}
\end{equation*}
then OMP will recover arbitrary $K$-sparse signal $x$ from $y =
\Phi x$ in exactly $K$ iterartions. The stability of recovery via
OMP in the term of coherence of $\Phi$ has been studied in
\cite{GMS}, \cite{Tr}, \cite{DET}, \cite{DET1}, \cite{TZh},
\cite{L10}. Recently M.~Davenport and M.~Wakin~\cite{DW}, and
E.~Liu and V.N.~Temlyakov \cite{LiuT} showed that
 if $\Phi$ satisfies RIP of order
$K+1$ with isometry constant
\begin{equation*}
\dt=\frac{1}{ 3K^{1/2}}(\cite{DW}),\quad \dt=\frac{1}{
(1+2^{1/2})K^{1/2}}(\cite{LiuT}),
\end{equation*}
 then OMP recovers arbitrary
$K$-sparse signal $x\in\R^N$ in exactly $K$ iterations.

To compare these results we recall estimates on coherence and RIP
for normalized random Bernoulli matrices $\Phi$ (each entry is
$\pm M^{-1/2}$ with probability $1/2$). For rather big $c_\mu$ we
have with high probability that
\begin{equation}\label{intr-mu1}
\mu(\Phi)\le c_\mu M^{-1/2}\log^{1/2} N.
\end{equation}
R.~Baraniuk,  M.~Davenport R.~Devore and M.~Wakin~\cite{BDDW} (see
also earlier B.S.~Kashin's work~\cite{K}) showed that random
Bernoulli matrix $\Phi$ with high probability satisfy RIP of order
$K$ with isometry constant $\dt$ for
\begin{equation}\label{intr-rip}
M = O\left(\frac{K\log(N/K)}{\dt^2}\right).
\end{equation}
Thus both results require $M=O(K^2)$ measurements for recovery of
$K$-sparse signal. The aim of this article to show that OMP can
recover sparse signals by essentially less number of measurements.

\begin{Theor}\label{Th-RIP} There exist absolute constants $C>0$ and $c>0$ such that if $\Phi$ satisfies the RIP of order $[CK^{1.2}]$
 with isometry constant $\dt =cK^{-0.2}$ and has coherence $\mu(\Phi)\le 1/(20 K^{0.8})$,
then for any $K$-sparse $x\in \R^N$, OMP will recover x exactly
from $y = \Phi x$ in at most $[CK^{1.2}]$ iterations.
\end{Theor}

Inequalities (\ref{intr-mu1}) and (\ref{intr-rip}) imply that for
rather big absolute constant $C_M>0$
 with high probability normalized random Bernoulli matrix $\Phi$
 with
\begin{equation*}
M = \left[C_M K^{1.6}\log N\right]
\end{equation*}
satisfies condition of Theorem~\ref{Th-RIP}.

Much less is known about the lower estimates.  H.~Rauhut~\cite{R}
proves that if $M\le \wt c K^{3/2}$ then for most random $M\times
N$ matrices there exists a $K$-sparse signal $x\in\R^N$ that can
not be recovered via $K$ iterations of OMP. Moreover, it's
conjectured in~\cite{R} (see also \cite{DM}) that for $M\le \wt
c_n K^{2-1/n}$, $n\in\N$, with high probability there exists a
$K$-sparse signal $x\in\R^N$ that can not be recovered via $K$
iterations of OMP from $y=\Phi x$.

\section{Auxiliary lemmas.}

We use two results on the rate of convergence of Orthogonal Greedy
Algorithm (OMP).

{\bf Theorem A. (R.A.~Devore, V.N.~Temlyakov,~\cite{DT})} {\it
Suppose that $y=\Phi x$. Then for any $l\ge 1$ we have
\begin{equation*}
\|r^{l}\| \le |x|_1 l^{-1/2}.
\end{equation*}
}

{\bf Theorem B. (EL,~\cite{L10})} {\it For any $l$, $1\le l\le
1/(20\mu(\Phi))$ we have
\begin{equation*}
\|r^{2l}\| \le 3\s_l(y,\Phi).
\end{equation*}
}

\noindent For $l\ge 0$ we set
\begin{equation*}
z^l := x - x^l.
\end{equation*}
Then by definition of OMP
\begin{equation}\label{rlPhizl}
r^l = y - \Phi x^l = \Phi x - \Phi x^l = \Phi z^l,\ l\ge 0.
\end{equation}
Assume that
\begin{equation*}
x = (x_1,\ldots, x_N),\quad z^l = (z^l_1,\ldots,z^l_N),\ l\ge 0.
\end{equation*}
Set
\begin{equation}\label{V0}
V_0 = \supp x,\quad \sharp V_0 \le K.
\end{equation}
 By $x|_V$, $V\subset V_0$, denote an element $(\wt
x_1,\ldots,\wt x_n)$ of $\R^N$ such that $\wt x_i = x_i$,  $i\in
V$, and $\wt x_i = 0$, $i\not\in V$. For each $V\subset V_0$ we
define
\begin{equation*}
R(V) = \sum_{i\in V} x_i^2.
\end{equation*}
Let
\begin{equation}\label{Cc}
 C:= 2\times 10^5,\quad c:=
10^{-6},\quad \dt := cK^{-0.2}.
\end{equation}

\begin{Lemm}\label{Lm-RIPsimple}
Suppose that $l+K\le CK^{1.2}$. Then we have
\begin{equation}\label{RIPsimple1}
\sum_{i\in \La^l}(z^l_i)^2\le 3\dt R(V_0\setminus \La^l),
\end{equation}
\begin{equation}\label{RIPsimple2}
R(V_0\setminus \La^l) \le (1+2\dt)\|r^l\|^2.
\end{equation}
\end{Lemm}
\begin{proof}
It's clear that $|z^l|_0\le |x|_0 + |x^l|_0\le K+l\le CK^{1.2}$,
so by RIP and (\ref{rlPhizl}) we have
\begin{equation}\label{RIPSimpleT1}
(1-\dt)\sum_{i=1}^N (z^l_i)^2 \le \|\Phi z^l\|^2 = \|r^l\|^2\le
(1+\dt)\sum_{i=1}^N (z^l_i)^2.
\end{equation}
On the other hand, using definition of $R(\cdot)$, and RIP for
$x|_{V_0\setminus \La^l}$ we write
\begin{equation*}
\|x|_{V_0\setminus \La^l}\|^2=R(V_0\setminus \La^l),
\end{equation*}
\begin{equation}\label{RIPSimpleT2}
(1-\dt)R(V_0\setminus \La^l)\le\left\|\Phi\left( x|_{V_0\setminus
\La^l}\right)\right\|^2\le (1+\dt)R(V_0\setminus \La^l).
\end{equation}
The definition of OMP implies that
\begin{equation*}
\|\Phi z^l\|^2 = \|r^l\|^2\le \left\|\Phi\left( x|_{V_0\setminus
\La^l}\right)\right\|^2.
\end{equation*}
Therefore using (\ref{RIPSimpleT1}) and (\ref{RIPSimpleT2}) we
have
\begin{equation*}
(1-\dt)\sum_{i=1}^N (z^l_i)^2 \le  \|r^l\|^2 \le \left\|\Phi\left(
x|_{V_0\setminus \La^l}\right)\right\|^2\le (1+\dt)R(V_0\setminus
\La^l),
\end{equation*}
\begin{equation*}
(1-\dt)\left(\sum_{i\in\La^l} (z^l_i)^2 + \sum_{i\in V_0\setminus
\La^l} (z^l_i)^2\right) \le  (1+\dt)R(V_0\setminus \La^l).
\end{equation*}
\begin{equation*}
\sum_{i\in\La^l} (z^l_i)^2 + R(V_0\setminus \La^l) \le
\frac{1+\dt}{1-\dt}R(V_0\setminus \La^l).
\end{equation*}
\begin{equation*}
\sum_{i\in\La^l} (z^l_i)^2 \le (\frac{1+\dt}{1-\dt}
-1)R(V_0\setminus \La^l) \le 3\dt R(V_0\setminus \La^l).
\end{equation*}
This completes the proof of  (\ref{RIPsimple1}). From
(\ref{RIPSimpleT1}) it follows that
\begin{equation*}
R(V_0\setminus \La^l) =  \sum_{i\in V_0\setminus\La^l} (z^l_i)^2
\le \sum_{i=1}^N (z^l_i)^2 \le (1-\dt)^{-1}\|r^l\|^2\le
(1+2\dt)\|r^l\|^2.
\end{equation*}
\end{proof}

For increasing sequence  $0=l_0 < l_1 <\cdots < l_s$, $s\ge 1$, we
denote
\begin{equation}\label{VkRk-def}
V_k:= V_0\setminus\La^{l_k},\ R_k = R(V_k),\ 0\le k\le s.
\end{equation}

\begin{Lemm}\label{Lm-OGArate1}
Suppose that $l_k + K \le CK^{1.2}$, $1\le k\le s$. Then for
arbitrary $p\in \N$ we have
\begin{equation*}
\|r^{l_k+p}\|^2\le \frac{R_k}{p}\left(6\dt CK^{1.2}+2K\right).
\end{equation*}
\end{Lemm}
\begin{proof}
Since $r^{l_k}=\Phi z^{l_k}$ we estimate by Theorem A.
\begin{equation}\label{fromThA}
\|r^{l_k+p}\|^2\le \frac{|z^{l_k}|_1^2}{p}.
\end{equation}
So to prove the lemma it's sufficient to estimate
\begin{equation*}
|z^{l_k}|_1^2=\left(\sum_{i = 1}^N |z^{l_k}_i|\right)^2 =
\left(\sum_{i\in V_0\cup \La^{l_k}} |z^{l_k}_i|\right)^2\le
2\left(\left(\sum_{i\in V_0\setminus\La^{l_k}}
|z^{l_k}_i|\right)^2 + \left(\sum_{i\in\La^{l_k}}
|z^{l_k}_i|\right)^2\right).
\end{equation*}
Applying (\ref{VkRk-def}) and (\ref{V0}) we have
\begin{equation}\label{OGArate1T1}
\left(\sum_{i\in V_0\setminus\La^{l_k}} |z^{l_k}_i|\right)^2 =
\left(\sum_{i\in V_0\setminus\La^{l_k}} |x_i|\right)^2 =
\left(\sum_{i\in V_k} |x_i|\right)^2 \le \sharp V_k \sum_{i\in
V_k} |x_i|^2 \le \sharp V_0 R_k\le R_kK.
\end{equation}
Using (\ref{RIPsimple1}) from Lemma~\ref{Lm-RIPsimple} we get
\begin{equation*}
\left(\sum_{i\in\La^{l_k}} |z^{l_k}_i|\right)^2\le \sharp\La^{l_k}
\sum_{i\in\La^{l_k}} (z^{l_k}_i)^2 = l_k \sum_{i\in\La^{l_k}}
(z^{l_k}_i)^2\le CK^{1.2} 3\dt R(V_0\setminus \La^{l_k})\le
CK^{1.2} 3\dt R_k.
\end{equation*}
Combining with (\ref{OGArate1T1}) we obtain the desirable
inequality
\begin{equation*}
|z^{l_k}|_1^2 \le R_k(6\dt CK^{1.2}+2K).
\end{equation*}
This together with (\ref{fromThA}) completes the proof of the
lemma.
\end{proof}

\begin{Lemm}\label{Lm-OGArate2}
Let $1\le p\le K^{0.8}$ and $l_k+2p\le CK^{1.2}$, $1\le k\le s$.
Then for any $W\subset V_k$ such that $\sharp W = p$ we have
\begin{equation*}
R(V_k\setminus\La^{l_k+2p})\le 10 R(V_k\setminus W)+ 30\dt R_k.
\end{equation*}
\end{Lemm}
\begin{proof}
According to RIP, (\ref{rlPhizl}), (\ref{RIPsimple1}) and
(\ref{VkRk-def}) we estimate
\begin{multline*}
\left(\s_p(r^{l_k})\right)^2\le \|r^{l_k}-\Phi(x|_W)\|^2 =
\|\Phi(z^{l_k})-\Phi(z^{l_k}|_W)\|^2 = \\
=\|\Phi(z^{l_k}- z^{l_k}|_W)\|^2 \le (1+\dt)\sum_{1\le i\le N,\
i\not\in
W}(z^{l_k}_i)^2 \le\\
\le (1+\dt)\left(\sum_{i\in V_k\setminus W}(z^{l_k}_i)^2 +
\sum_{1\le i\le N,\ i\not\in V_k}(z^{l_k}_i)^2\right)\le
(1+\dt)\left(\sum_{i\in V_k\setminus W}(x_i)^2 + \sum_{i\in
\La^{l_k}}(z^{l_k}_i)^2\right)\le\\
\le (1+\dt)(R(V_k\setminus W) + 3\dt R(V_0\setminus\La^{l_k})) \le
(1+\dt)(R(V_k\setminus W) + 3\dt R_k).
\end{multline*}
Since
\begin{equation*}
p\le K^{0.8}\le 1/(20\mu(\Phi))
\end{equation*}
we can apply Theorem~B and get
\begin{equation*}
\|r^{l^k+2p}\|\le 3 \s_p(r^{l_k}).
\end{equation*}
Using (\ref{RIPsimple2}) from Lemma~\ref{Lm-RIPsimple} we obtain
\begin{multline*}
R(V_k\setminus\La^{l_k+2p}) = R(V_0\setminus \La^{l_k + 2p}) \le
(1+2\dt)\|r^{l_k + 2p}\|^2\le(1+2\dt)9\left(\s_p(r^{l_k})\right)^2\le\\
\le (1+2\dt)9 (1+\dt)(R(V_k\setminus W) + 3\dt R_k)\le 10\left(
R(V_k\setminus W)+ 3\dt R_k\right).
\end{multline*}
\end{proof}

\section{Proof of Theorem~\ref{Th-RIP}.}

We prove by induction on $k$ that if $l_k + K \le CK^{1.2}$ and
$V_k\ne\emptyset$, then we can define $l_{k+1} > l_k$,  $V_{k+1}$
and $R_{k+1}$ satisfying (\ref{VkRk-def}) such that
\begin{equation}\label{CheckBelow}
l_{k+1} + K \le CK^{1.2},
\end{equation}
 and at least one of the following
statements hold
\begin{equation}\label{caseA}
\text{(A) }\quad l_{k+1}-l_k \le 3\times 10^{4} (6\dt
CK^{1.2}+2K),\quad \sharp(V_k\setminus V_{k+1}) \ge K^{0.8};
\end{equation}
\begin{equation}\label{caseB}
\text{(B) }\quad l_{k+1}-l_k \le 2 K^{0.8},\quad
\sharp(V_k\setminus V_{k+1}) \ge K^{0.6};
\end{equation}
 \begin{equation}\label{caseC}
\text{(C) }\quad l_{k+1}-l_k \le 2K^{0.6},\quad
\sharp(V_k\setminus V_{k+1}) \ge K^{0.4};
\end{equation}
\begin{equation}\label{caseD}
\text{(D) }\quad l_{k+1}-l_k \le 2K^{0.4},\quad
\sharp(V_k\setminus V_{k+1}) \ge K^{0.2};
\end{equation}
\begin{equation}\label{caseE}
\text{(E) }\quad l_{k+1}-l_k \le 2K^{0.2},\quad
\sharp(V_k\setminus V_{k+1}) \ge 1.
\end{equation}

Set
\begin{equation*}
p_A = [3\times 10^{4} (6\dt CK^{1.2}+2K)],\quad W_A = V_k \cap
\La^{l_k+p_A}.
\end{equation*}
If $\sharp W_A\ge K^{0.8}$ we define
\begin{equation*}
l_{k+1}:=l_k+p_A,\quad V_{k+1} := V_0\setminus \La^{l_{k+1}} =
V_k\setminus W_A,\quad R_{k+1}:=R(V_{k+1}).
\end{equation*}
Then $\sharp(V_k\setminus V_{k+1}) = \sharp W_A$ and
statement~(\ref{caseA}) holds. The inequality
\begin{equation*}
l_k+ p_A + K \le CK^{1.2}
\end{equation*}
will be checked below. Assume that
\begin{equation}\label{WA}
  \sharp W_A < K^{0.8}
  \end{equation}
 Applying Lemma~\ref{Lm-OGArate1}
we have
\begin{equation*}
\|r^{l_{k}+p_A}\|^2\le \frac{R_k(6\dt CK^{1.2}+2K)} {[3\times
10^{4} (6\dt CK^{1.2}+2K)]}\le \frac{R_k} {2.5\times 10^{4}}.
\end{equation*}
Using (\ref{RIPsimple2}) from Lemma~\ref{Lm-RIPsimple} we estimate
\begin{equation}\label{RVkWA}
R(V_k\setminus W_A) = R(V_0\setminus \La^{l_{k}+p_A}) \le
 (1+2\dt)\|r^{l_{k}+p_A}\|^2\le (1+2\dt)\frac{R_k} {2.5\times 10^{4}}\le \frac{R_k} {2\times 10^4}.
\end{equation}
Set
\begin{equation*}
 p_B:=\sharp W_A,\quad  W_B = V_k \cap \La^{l_k +2p_B}.
\end{equation*}
If $\sharp W_B\ge K^{0.6}$ we set
\begin{equation*}
l_{k+1}=l_k+2p_B,\quad V_{k+1}=
V_k\setminus\La^{l_{k+1}}=V_k\setminus W_B,\quad
R_{k+1}=R(V_{k+1}).
\end{equation*}
Then $\sharp(V_k\setminus V_{k+1}) = \sharp W_B\ge K^{0.6}$ and
taking into account (\ref{WA}) we obtain (\ref{caseB}). The
inequality
\begin{equation*}
l_k+2p_B+K\le CK^{1.2}
\end{equation*}
will be checked below.

Assume that
\begin{equation}\label{WB}
  \sharp W_B < K^{0.6}.
  \end{equation}
Applying Lemma~\ref{Lm-OGArate2} for $W=W_A$ and $p=p_B$, and
inequality (\ref{RVkWA}) we get
\begin{equation}\label{RVkWB}
R(V_k\setminus W_B) = R(V_k\setminus\La^{l_k+2p_B})\le 10
R(V_k\setminus W_A)+ 30\dt R_k\le \frac{R_k} {2\times 10^3} +
30\dt R_k.
\end{equation}

We repeat these calculations three more times.

Set
\begin{equation*}
p_C:=\sharp W_B,\quad W_C = V_k \cap \La^{l_k +2p_C}.
\end{equation*}
If $\sharp W_C\ge K^{0.4}$ we set
\begin{equation*}
l_{k+1}=l_k+2p_C,\quad
V_{k+1}=V_k\setminus\La^{l_{k+1}}=V_k\setminus W_C,\quad
R_{k+1}=R(V_{k+1}).
\end{equation*}
Then $\sharp(V_k\setminus V_{k+1}) = \sharp W_C\ge K^{0.4}$ and
taking into account (\ref{WB}) we obtain (\ref{caseC}). The
inequality
\begin{equation*}
l_k+2p_C+K\le CK^{1.2}
\end{equation*}
will be checked below.

Assume that
\begin{equation}\label{WC}
  \sharp W_C < K^{0.4}.
  \end{equation}
Applying Lemma~\ref{Lm-OGArate2} for $W=W_B$ and $p=p_C$, and
inequality (\ref{RVkWB}) we get
\begin{equation}\label{RVkWC}
R(V_k\setminus W_C) = R(V_k\setminus\La^{l_k+2p_C})\le 10
R(V_k\setminus W_B)+ 30\dt R_k\le \frac{R_k} {2\times 10^2} +
3(10^2+10)\dt R_k.
\end{equation}

Set
\begin{equation*}
p_D:=\sharp W_C,\quad W_D = V_k \cap \La^{l_k +2p_D}.
\end{equation*}
If $\sharp W_D\ge K^{0.2}$ we set
\begin{equation*}
l_{k+1}=l_k+2p_D,\quad
V_{k+1}=V_k\setminus\La^{l_{k+1}}=V_k\setminus W_D,\quad
R_{k+1}=R(V_{k+1}).
\end{equation*}
Then $\sharp(V_k\setminus V_{k+1}) = \sharp W_D\ge K^{0.2}$ and
taking into account (\ref{WC}) we obtain (\ref{caseD}). The
inequality
\begin{equation*}
l_k+2p_D+K\le CK^{1.2}
\end{equation*}
will be checked below.

Assume that
\begin{equation}\label{WD}
  \sharp W_D < K^{0.2}.
  \end{equation}
Applying Lemma~\ref{Lm-OGArate2} for $W=W_C$ and $p=p_D$, and
inequality (\ref{RVkWC}) we get
\begin{equation}\label{RVkWD}
R(V_k\setminus W_D) = R(V_k\setminus\La^{l_k+2p_D})\le 10
R(V_k\setminus W_C)+ 30\dt R_k\le \frac{R_k} {20} + 3(10^3 + 10^2
+ 10)\dt R_k.
\end{equation}

Set
\begin{equation*}
p_E:=\sharp W_D,\quad W_E = V_k \cap \La^{l_k +2p_E},
\end{equation*}
\begin{equation*}
l_{k+1}=l_k+2p_E,\quad
V_{k+1}=V_k\setminus\La^{l_{k+1}}=V_k\setminus W_E,\quad
R_{k+1}=R(V_{k+1}).
\end{equation*}
Taking into account (\ref{WD}) we get
\begin{equation*}
l_{k+1}- l_k \le 2K^{0.2}.
\end{equation*}
The inequality
\begin{equation*}
l_k+2p_E+K\le CK^{1.2}
\end{equation*}
will be checked below.

Applying Lemma~\ref{Lm-OGArate2} for $W=W_D$ and $p=p_E$, and
inequalities (\ref{RVkWD}) and (\ref{Cc}) we have
\begin{multline*}
R_{k+1} = R(V_k\setminus W_E) = R(V_k\setminus\La^{l_k+2p_E})\le
10 R(V_k\setminus W_D)+ 30\dt R_k\le\\
\le \frac{R_k} {2} + 3(10^4 + 10^3 + 10^2 + 10)\dt R_k < R_k(1/2 +
4\times 10^4\dt) < R_k.
\end{multline*}
Therefore $W_E\ne\emptyset$ and $\sharp(V_k\setminus V_{k+1})\ge
1$ and statement (\ref{caseE}) holds.

Thus to complete the proof of induction assumption it remains to
estimate
\begin{equation*}
l_k+p_A + K,\ l_k+2p_B + K,\ l_k+2p_B + K,\ l_k+2p_C + K,\
l_k+2p_D + K.
\end{equation*}
It's clear that the biggest of these numbers is the first one.

Using induction assumption, inclusion $V_0\supset
V_1\supset\cdots\supset V_k$, and the equality $\sharp V_0 = K$ we
claim that statement (\ref{caseA}) (case A) can be fulfilled not
more than $K^{0.2} = K/K^{0.8}$ times, statement (\ref{caseB})
(case B) not more than $K^{0.4} = K/K^{0.6}$ times, statement
(\ref{caseC}) (case C) not more than $K^{0.6} = K/K^{0.4}$ times,
statement (\ref{caseD}) (case D) not more than $K^{0.8} =
K/K^{0.2}$ times, and statement (\ref{caseC}) (case E) not more
than $K$ times. Hence, using (\ref{Cc}), we obtain
\begin{multline*}
l_k \le 3\times 10^4 (6\dt CK^{1.2}+2K)K^{0.2} + 2K^{0.8}K^{0.4} +
2K^{0.6}K^{0.6} + 2K^{0.4}K^{0.8} + 2K^{0.2}K = \\
=(18\times 10^4 \dt C K^{0.2} + 6 \times 10^4 + 8)K^{1.2},
\end{multline*}
\begin{multline*}
l_k+p_A + K \le (18\times 10^4 \dt C K^{0.2} + 6 \times 10^4 +
8)K^{1.2} + 3\times 10^4 (6\dt CK^{1.2}+2K) + K\le\\
\le (36\times 10^4 \dt C K^{0.2} + 12 \times 10^4 + 9)K^{1.2}\le\\
\le
(36\times 10^4 \times 10^{-6}K^{-0.2} 2\times 10^5 K^{0.2} + 12
\times 10^4 + 9)K^{1.2} < 20\times 10^4K^{1.2} = CK^{1.2}.
\end{multline*}

Since $\sharp(V_k\setminus V_{k+1}) \ge 1$ there exists $s\in\N$
such that $V_s=\emptyset$ and $l_s+K\le CK^{1.2}$. Therefore by
(\ref{VkRk-def}) and (\ref{rlPhizl}) we have
\begin{equation*}
V_0\subset \La^{l_s}
\end{equation*}
and
\begin{equation*}
r^{l_s} = \Phi z^{l_s} = \Phi(x-x^{l_s}) = 0.
\end{equation*}
Using RIP we finally obtain that
\begin{equation*}
x = x^{l_s}.\quad \square
\end{equation*}

\begin{Remar}
We guess that constant $3$ in Theorem B is not optimal and hence
constants $C$ and $c^{-1}$ from (\ref{Cc}) can be reduced.
\end{Remar}

The author thanks professor V.N.~Temlyakov and professor
S.V.~Konyagin for useful discussions.

\end{document}